\documentclass{amsart}
\title[Global Regular Solutions in $\Phi\left(2\right)$]{Global Regular Solutions
for the Navier-Stokes system with small initial data in $\Phi\left(2\right)$: an elementary approach.}
\author{Jean Cortissoz}
\address{Departamento de Matem\'aticas,
Universidad de Los Andes\\
Bogot\'a DC, COLOMBIA}
\email{jcortiss@uniandes.edu.co}
\subjclass{35Q30}
\keywords{Navier-Stokes equations, Regularity}
\newtheorem{theorem}{Theorem}
\newtheorem{lemma}{Lemma}
\newtheorem{corollary}{Corollary}

\begin{document}
\begin{abstract}
In this note we show that solutions of the Navier Stokes equation that 
are small in $L^{\infty}\left(\left(0,T\right), \Phi\left(2\right)\right)$
are also globally smooth. This result
relates to a recent result in \cite{Arnold} (announced in \cite{Kaloshin}).
Also, we give elementary proofs of some classical
results of Giga \cite{Y. Giga}, v. Wahl \cite{Wahl}, and Kozono-Sohr \cite{Sohr}. 
\end{abstract}
\maketitle

\section{Introduction}

One of the outstanding problems in mathematics
is the existence of global regular solutions to the Navier-Stokes system
\begin{equation}
\label{Navierstokes}
\left\{
\begin{array}{l}
u_t-\Delta u+u\cdot\nabla u+\nabla p=0 \quad \mbox{in} \quad \mathbb{T}^3\times \left(0,\infty\right)\\
u\left(x,0\right)=\psi, \quad div \,u=0,
\end{array}
\right.
\end{equation}
where $\mathbb{T}^3=\left[0,1\right]^3$, with periodic boundary conditions.

Several short time existence results and small-initial data-global existence results 
have been shown for different Banach spaces and the literature on the subject is extensive.
This note is another contribution to the subject.

Before we state our main results we must give a few definitions.
Given a function $u\left(x,t\right)\in L^2\left(\mathbb{T}^3\right)$ we write its Fourier expansion as 
\[
\sum_{\mathbf{k}}u_{\mathbf{k}}\left(t\right)\exp\left(2\pi i \left<x,\mathbf{k}\right>\right),
\quad \mathbf{k}=\left(k_1,k_2,k_3\right)\in\mathbb{Z}^3,
\]
and we define the spaces $\Phi\left(\alpha\right)\subset \mathcal{P}'$ ($\mathcal{P}'$ is
the dual of the space of $C^{\infty}$ periodic functions on $\mathbb{T}^3$) as follows,
\[
\Phi\left(\alpha\right)=
\left\{f\,:\, \left|f_{\mathbf{k}}\right|\leq\frac{c}{\left|\mathbf{k}\right|^{\alpha}}
,\,\mathbf{k}\neq 0, \quad f_0=0\right\}
\]
endowed with the norm
\[
\left\|f\right\|_{\alpha}= 
\sup_{\mathbf{k}\in \mathbb{Z}^3\setminus \left\{0\right\}} \left|\mathbf{k}\right|^{\alpha}\left|f_{\mathbf{k}}\right|
\]
which makes them Banach spaces. 
Notice that if $\alpha >\frac{3}{2}$, $\Phi\left(\alpha\right)\subset L^2\left(\mathbb{T}^3\right)$.
These spaces are considered in \cite{Arnold}, where 
the following Theorem with $\alpha=2+\epsilon$, $\epsilon>0$, is proved
\begin{theorem}
Let $0<3\epsilon<1$ and $\left\|\psi\right\|_{\alpha}\leq \delta$ where 
$\psi=\frac{c_0\left(\mathbf{k}\right)}{\left|\mathbf{k}\right|^{\alpha}}$ is the initial condition
and $\delta=\delta\left(\alpha\right)$ is sufficiently small. Then equation
(\ref{Navierstokes}) has a global solution
 $v\left(t,\mathbf{k}\right)=\frac{c\left(t,\mathbf{k}\right)}{\left|\mathbf{k}\right|^{\alpha}}$ 
such that $c\left(t,\mathbf{k}\right)$ is a continuous mapping from $\left[0,\infty\right)$ to
$L^{\infty}\left(\mathbb{Z}^3\setminus\left\{0\right\}\right)$.
\end{theorem}

The result with $\alpha=2$ is announced in \cite{Kaloshin}. In this note we prove some related results,
where the treat for the reader is the elementarity of the proofs, which are in the
spirit of the ideas presented in \cite{Sinai} (in this note we use nothing beyond the Cauchy-Schwarz
inequality). 

We shall show the following results,

\begin{theorem}
\label{uniqueness}
There is $\epsilon>0$ such that if a Leray-Hopf solution $u\left(x,t\right)$ of the Navier Stokes system
satisfies 
\begin{equation}
\label{smallnessassumption}
\sup_{\left|\mathbf{k}\right|\geq K}
\left|\mathbf{k}\right|^2\left|u_{\mathbf{k}}\left(t\right)\right|<\epsilon \quad\mbox{on} \quad \left(0,T\right)
\end{equation}
for some $K>0$ independent of time, 
then $u$ is smooth for every $0<t<T$.
\end{theorem}

\begin{theorem}
\label{smallnorm}
There exists an $\epsilon>0$ such that if the initial condition $\psi$ satisfies 
\[
\left\|\psi\right\|_2 < \epsilon
\] 
then any Leray-Hopf solution of (\ref{Navierstokes}) with initial
condition $\psi$ satisfies
\[
\left\|u\left(t\right)\right\|_2<\epsilon \quad\mbox{if} \quad t>0.
\]
\end{theorem}

Recall that a Leray-Hopf solution of (\ref{Navierstokes}) 
with initial data $\psi\in L^2\left(\mathbb{T}^3\right)$ is a function 
$u:\left[0,T\right) \longrightarrow L^2\left(\mathbb{T}^3\right)$ such that
\begin{enumerate}

\item
$u$ is weakly continuous;

\item
\[
u\in L^{\infty}\left(0,T;L^2\left(\mathbb{T}^3\right)\right)\cap L^2\left(0,T;H^1\left(\mathbb{T}^3\right)\right);
\]

\item
$u$ satisfies
\begin{eqnarray*}
\left<u\left(t\right),\varphi\left(t\right)\right>
+\int_0^t 
-\left<u\left(\tau\right),\frac{\partial}{\partial \tau}\varphi\left(\tau\right)\right>
+\left<\nabla u\left(\tau\right),\nabla \varphi \left(\tau\right)\right>+
\left<u\cdot \nabla u,\varphi\right>\, d\tau\\
=\left<u_0,\varphi\left(0\right)\right>
\end{eqnarray*}
for all $\varphi\in C^{\infty}\left(\mathbb{T}^3\times\left[0,T\right)\right)$ with $div \varphi=0$;
and 
\item

the energy inequality
\[
\left\|u\left(t\right)\right\|^2_{L^2\left(\mathbb{T}^3\right)}+
2\int_0^t \left\|\nabla u\left(\tau\right)\right\|^2_{L^2\left(\mathbb{T}^3\right)}\,d\tau\leq 
\left\|\psi\right\|^2_{L^2\left(\mathbb{T}^3\right)}
\]
holds.

\end{enumerate}

For results on the existence of Leray-Hopf solutions of the Navier-Stokes system
with initial data in $L^2\left(\mathbb{T}^3\right)$, the reader may consult
Chapter 3 of 
\cite{Temam}.

The following result
is an inmediate consequence of the existence of a Leray-Hopf solution 
for a given initial data $\psi\in L^2\left(\mathbb{T}^3\right)$, and Theorems \ref{uniqueness} and \ref{smallnorm}.
 
\begin{corollary}
\label{globalsolution}
there is $\epsilon>0$ such that if the initial condition $\psi$ satisfies
\[
\left\|\psi\right\|_2<\epsilon
\]
then there is
a global regular solution to problem (\ref{Navierstokes}).
\end{corollary}

It must pointed out that in Corollary \ref{globalsolution},
as it will be clear later on, one does not require
the initial value $\psi$ to be real valued, nor the 
use of Leray-Hopf's Existence Theorem to prove it, as we have
suggested. we postpone a discussion of this issue 
until the final Section of this paper (see Section \ref{corolario1}). 

Also, as another 
application of the methods used in this note, we show the following "classical" regularity result,

\begin{theorem}
\label{Escauriaza}
Let $u\left(x,t\right)\in L^{\infty}\left(0,T;H^{\frac{1}{2}}\left(\mathbb{T}^3\right)\right)$ be
a solution of (\ref{Navierstokes}). There exists a $\delta>0$ such that for any $\rho>0$,
if
$\left\|u\left(t\right)\right\|_{H^{\frac{1}{2}}\left(\mathbb{T}^3\right)}< \delta$ for $t\in \left(0,T\right)$, then
there is a $K=K\left(\left\|u\right\|_{L^{\infty}\left(0,T;H^{\frac{1}{2}}\left(\mathbb{T}^3\right)\right)},\rho\right)$ 
such that if 
$\left|\mathbf{k}\right|\geq K$ and $t>\rho$, then
\[
\left|u_{\mathbf{k}}\left(t\right)\right|\leq 
\frac{\left(2\left(4+\sqrt{2}\right)L_0\right)/\left(1-2c\sqrt{L_0}\right)}{\left|\mathbf{k}\right|^2},
\]
where 
$\sqrt{L_0}=\sup_{t\in\left(0,T\right)} \left\|u\left(t\right)\right\|_{H^{\frac{1}{2}}\left(\mathbb{T}^3\right)}$
and $c$ is a universal constant.
Therefore if $L_0$ is small enough, $u$ is regular on $\left(0,T\right)$.
\end{theorem}

As a consequence of the proof of Theorem \ref{Escauriaza} one can show that a solution to the Navier-Stokes
system that belongs to the space $C\left(\left(0,T\right),H^{\frac{1}{2}}\left(\mathbb{T}^3\right)\right)$ is regular
(the same result, but with $L^3\left(\Omega\right)$,
$\Omega$ a domain with $C^{2+\mu}$ boundary, instead of $H^{\frac{1}{2}}\left(\mathbb{T}^3\right)$ was proved 
by Giga in \cite{Y. Giga} and by von Wahl in \cite{Wahl}), and also that
small discontinuities in $H^{\frac{1}{2}}\left(\mathbb{T}^3\right)$ norm are allowed
 (for the related result on $L^3\left(\Omega\right)$ see \cite{Sohr});
we indicate how this can be done in Section \ref{Seregin2}. 
We must also add that the results of this paper can be generalized to higher dimensions 
(of course with the appropiate obvious hypothesis) without too much effort.

This paper is organized as follows. In Section \ref{seccion2} we give a proof of Theorem \ref{smallnorm};
in Section \ref{seccion3} we give a proof of Theorem \ref{uniqueness}; in Section \ref{Seregin} we give a
proof of Theorem \ref{Escauriaza}; and in Section \ref{lastremarks} we make further comments on the
results of this paper. 

\subsection{Some Remarks and Notation}

The Navier Stokes system can be written in the phase space as follows
\[
u_{\mathbf{k},t}^m= -\left|\mathbf{k}\right|^2 u_{\mathbf{k}}^m
-i\sum k_j u_{\alpha}^j u_{\mathbf{k}-\alpha}^m+
i\sum \frac{k_m k_l k_j}{\left|\mathbf{k}\right|^2}u_{\alpha}^l u_{\mathbf{k}-\alpha}^j 
\]

Notice that by the divergence-free property, one also has that
the infinite dimensional 
ODE system for the Fourier coefficients of the Navier-Stokes equation
can be written as
\[
u_{\mathbf{k},t}^m= -\left|\mathbf{k}\right|^2 u_{\mathbf{k}}^m
-i\sum \alpha_j u_{\alpha}^m u_{\mathbf{k}-\alpha}^j+
i\sum \frac{k_m k_l \alpha_j}{\left|\mathbf{k}\right|^2}u_{\alpha}^l u_{\mathbf{k}-\alpha}^j 
\]

From now on, since all that matters is its assymptotic behavior, we will write the sums
\[
\sum k_j u_{\alpha}^j u_{\mathbf{k}-\alpha}^m , \quad
\sum \frac{k_m k_l k_j}{\left|\mathbf{k}\right|^2}u_{\alpha}^l u_{\mathbf{k}-\alpha}^j \quad \mbox{as}\quad
\sum \mathbf{k} u_{\alpha}u_{\mathbf{k}-\alpha},
\]
\[
\mbox{and}\quad
\sum \alpha_j u_{\alpha}^m u_{\mathbf{k}-\alpha}^j,\quad
\sum \frac{k_m k_l \alpha_j}{\left|\mathbf{k}\right|^2}u_{\alpha}^l u_{\mathbf{k}-\alpha}^j
\quad \mbox{as}\quad 
\sum \alpha u_{\alpha} u_{\mathbf{k}-\alpha}
\]

If $X$ is a "classical" Banach space (like the $L^p$'s or $H^q$'s) we denote its norm by $\left\|\cdot\right\|_X$.
It is also important to notice the following: since the solutions to (\ref{Navierstokes}) are divergence free,
in our estimations sums of the form 
$\sum_{\alpha\in \mathcal{Z}} \alpha u_{\alpha}u_{\mathbf{k}-\alpha}$ are equivalent to sums
of the form $\sum_{\alpha\in \mathcal{Z}}\mathbf{k}u_{\alpha}u_{\mathbf{k}-\alpha}$.

\section{Proof of Theorem \ref{smallnorm}}
\label{seccion2}

Following the analysis and arguments in \cite{Sinai}, all we must show is that
the sum 
\[
\sum_{\alpha}\alpha u_{\alpha} u_{\mathbf{k}-\alpha}
\]
is small compared to the term $-\left|\mathbf{k}\right|^2 u_\mathbf{k}$ whenever
$\left|u_{\mathbf{k}}\right|$ is close to $\frac{\epsilon}{\left|\mathbf{k}\right|^2}$. Here
we work formally, but the arguments can be made rigorous by using Galerkin approximations and then taking limits.
Now that the reader has been warned, we proceed with our calculations.
Let 
\[
\sum_{\mathbf{\alpha}}\alpha u_{\alpha}u_{\mathbf{k}-\alpha}=I + II + III
\]
where
\[
I=\sum_{\left|\alpha\right|\leq 2\left|\mathbf{k}\right|, \left|\mathbf{k}-\alpha\right|\leq \frac{\left|\mathbf{k}\right|}{2}} 
\alpha u_{\alpha}u_{\mathbf{k}-\alpha},
\quad
II=\sum_{\left|\alpha\right|\leq 2\left|\mathbf{k}\right|, \left|\mathbf{k}-\alpha\right|> \frac{\left|\mathbf{k}\right|}{2}}
\alpha u_{\alpha}u_{\mathbf{k}-\alpha},
\]
and
\[
III=\sum_{\left|\alpha\right|>2\left|\mathbf{k}\right|}\mathbf{k} u_{\alpha}u_{\mathbf{k}-\alpha}.
\]

Under the assumption $\left|u_{\alpha}\right|\leq \frac{\epsilon}{\left|\alpha\right|^2}$, we can bound each of these 
terms as follows,
\[
\left|I\right|\leq 2\left|\mathbf{k}\right|\frac{4\epsilon}{\left|\mathbf{k}\right|^2}
\sum_{\left|\mathbf{k}-\alpha\right|\leq \frac{\left|\mathbf{k}\right|}{2}}
\left|u_{\mathbf{k}-\alpha}\right|
\leq \frac{4\epsilon}{\left|\mathbf{k}\right|}C\epsilon\left|\mathbf{k}\right|=4C\epsilon^2.
\]
where $C$ is a constant such that 
\begin{equation}
\label{integralbound1}
\sum_{1\leq\left|\alpha\right|\leq r} \frac{1}{\left|\alpha\right|^2}\leq 
C\int_{1\leq\left|\mathbf{x}\right|\leq r,\,\, \mathbf{x}\in \mathbf{R}^3}\frac{1}{\left|\mathbf{x}\right|^2}\,d\mathbf{x}.
\end{equation}

In a similar way, for a constant $C$ universally defined, we obtain the estimate
\[
\left|II\right|\leq 4C\epsilon^2.
\]

Finally, from
\[
\left|III\right|\leq \left|\mathbf{k}\right|\sum_{\left|\alpha\right|>2\left|\mathbf{k}\right|}
\frac{\epsilon^2}{\left|\alpha\right|^2\left|\mathbf{k}-\alpha\right|^2},
\]
using the elementary estimates
\[
\left|\alpha\right|\leq \left|\mathbf{k}-\alpha\right|+\left|\mathbf{k}\right|\quad
\mbox{and}\quad \left|\mathbf{k}-\alpha\right|\geq \left|\mathbf{k}\right|
\qquad (\mbox{recall that} \quad \left|\alpha\right| >2\left|\mathbf{k}\right|)
\]
we obtain,
\[
\left|III\right|\leq \left|\mathbf{k}\right|
\sum_{\left|\alpha\right|>2\left|\mathbf{k}\right|}\frac{4\epsilon^2}{\left|\alpha\right|^4}
\leq 4C\epsilon^2.
\] 
where $C>0$ is a constant such that
\begin{equation}
\label{integralbound2}
\sum_{\left|\alpha\right|\geq \left|\mathbf{k}\right|} \frac{1}{\left|\alpha\right|^4}\leq 
C\int_{\left|\mathbf{x}\right|\geq 
\left|\mathbf{k}\right|,\,\, \mathbf{x}\in \mathbf{R}^3}\frac{1}{\left|\mathbf{x}\right|^4}\,
d\mathbf{x}\leq \frac{c}{\left|\mathbf{k}\right|}.
\end{equation}

Therefore, for a constant $C>0$ universally defined, the following bound holds,
\[
\left|\sum_{\mathbf{\alpha}}\alpha u_{\alpha}u_{\mathbf{k}-\alpha}\right|\leq
12C\epsilon^2.
\]

From the last inequality it is clear that by taking $\epsilon>0$ sufficiently
small the result follows.

\hfill $\Box$

\section{Proof of Theorem \ref{uniqueness}}
\label{seccion3}

Let $\rho>0$, we will show that there exist an $\epsilon>0$ such that
if $u$ a Leray-Hopf solution fo the Navier-Stokes system satisfies
(\ref{smallnessassumption}), then $u$ is smooth for $t>\rho$. 
Let $\epsilon\in\left(0,\frac{1}{3}\right)$ and assume $u$ satisfies (\ref{smallnessassumption})
for this $\epsilon$.
Then, there is a $k_{-1}$ such that 
\[
\left|u_{\mathbf{k}}\left(t\right)\right|\leq \frac{\epsilon}{\left|\mathbf{k}\right|^2}
\quad \mbox{if}\quad \left|\mathbf{k}\right|\geq k_{-1},
\]
and without loss of generality we can assume that $k_{-1}$ is large enough so that
\begin{equation}
\label{exponentialestimate1}
\exp\left(-\frac{\rho\left|\mathbf{k}\right|^2}{2^{n+1}}\right)<\epsilon^{2^n} \quad
\mbox{holds whenever} \quad \left|\mathbf{k}\right|\geq \frac{k_{-1}}{\epsilon^{2^n}}, \quad 
\mbox{for all} \quad n\in \mathbb{N}.
\end{equation}

Notice that, by the energy inequality satisfied by Leray-Hopf solutions,
there exists a constant $D>0$ such that for all 
frequencies $\mathbf{k}$ such that $\left|\mathbf{k}\right|<k_{-1}$ it holds that 
\[
\left|u_{\mathbf{k}}\left(t\right)\right|\leq \frac{D}{\left|\mathbf{k}\right|^2}
\quad \mbox{for}\quad t>0.
\]

Choose $k_0$ so that $\frac{k_{-1}}{k_0}\cdot D < \epsilon<\frac{1}{2}$.
Define $k_i=\frac{1}{\epsilon^{2^i}}k_0$ 
and  a sequence $\mu_n$ inductively as follows
\[
\mu_0=1, \mu_1=2; \quad\mbox{and}\quad
\mu_{n+1}=2\mu_n-1, \quad\mbox{if} \quad n>1.
\]

The sequence thus defined satisfies the following,
\begin{lemma}
\label{exponentialrecurrence}
For the sequence $\mu_n$ it holds that
$
\frac{1}{2}<\frac{\mu_n}{2^n}<1
$, for $n\geq 2$.
\end{lemma}
\begin{proof}
It is easily proved by induction.
\end{proof}

Now we continue with the proof of Theorem \ref{uniqueness}. Define 
\[
t=\rho-\frac{\rho}{2^n},
\]
and assume that for $t>t_n$ the following estimate 
holds 
\[
\left|u_{\mathbf{k}}\left(t\right)\right|\leq \frac{\epsilon^{\mu_n}}{\left|\mathbf{k}\right|^2}
\quad\mbox{if} \quad \left|\mathbf{k}\right|\geq k_n.
\] 
The idea is to show that this estimate improves for frequencies $\mathbf{k}$
such that $\left|\mathbf{k}\right|\geq k_{n+1}$ and times $t>t_{n+1}$.
Notice that this estimate holds for $n=0$. 

In order to proceed, assume $\left|\mathbf{k}\right|\geq k_{n+1}$. As before, we must estimate the sum
\[
\sum_{\alpha} \alpha u_{\alpha}u_{\mathbf{k}-\alpha}=I+II+III+IV+V
\]
where the meaning of $I,II, III, IV$ and $V$ will become clear in what follows.
Before we start, we  must point out that the constant $C$ that appears in the following estimates is the maximum
between the constant  $C$ that appears in inequality (\ref{integralbound1}) and
the constant $C$ that appears in inequality (\ref{integralbound2}).

Let us then begin by estimating the first term
\[
I= I_1+I_2,
\]
\begin{eqnarray*}
\left|I_1\right|\leq
\sum_{0\leq\left|\alpha\right|\leq k_{-1},\left|\mathbf{k}-\alpha\right|\geq \frac{\left|\mathbf{k}\right|}{2}}
\left|\alpha u_{\alpha}u_{\mathbf{k}-\alpha}\right|
&\leq& k_{-1}\sum_{0\leq\left|\alpha\right|\leq k_{-1},\left|\mathbf{k}-\alpha\right|\geq \frac{\left|\mathbf{k}\right|}{2}}
\left|u_{\alpha}u_{\mathbf{k}-\alpha}\right|\\
&\leq& \frac{4k_{-1}\epsilon^{\mu_n}}{k_{n+1}^2}\sum_{\left|\alpha\right|\leq k_{-1}}\left|u_{\alpha}\right|\\
&&\mbox{(and by inequality (\ref{integralbound1}))}\\
&\leq& 4\epsilon^{\mu_n}D\cdot C\cdot \left(\frac{k_{-1}}{k_n}\right)^2 \left(\frac{k_n}{k_{n+1}}\right)^2,
\end{eqnarray*}
and since $k_0<k_n$, by our choices we obtain
\[
\left|I_1\right|\leq 4\epsilon^{\mu_n}\cdot\frac{1}{2}\cdot C\epsilon^{\mu_n}=2C\epsilon^{2\mu_n}.
\]
On the other hand

\begin{eqnarray*}
\left|I_2\right|&\leq&
\sum_{k_{-1}\leq\left|\alpha\right|\leq k_n,\left|\mathbf{k}-\alpha\right|\geq \frac{\left|\mathbf{k}\right|}{2}}
\left|\alpha u_{\alpha}u_{\mathbf{k}-\alpha}\right|\\
&\leq& k_n\sum_{k_{-1}\leq\left|\alpha\right|\leq k_n,\left|\mathbf{k}-\alpha\right|\geq \frac{\left|\mathbf{k}\right|}{2}}
\left|u_{\alpha}u_{\mathbf{k}-\alpha}\right|\\
&\leq& \frac{4k_n\epsilon^{\mu_n}}{k_{n+1}^2}\sum_{k_{-1}\leq \left|\alpha\right|\leq k_n}\left|u_{\alpha}\right|\\
&\leq& 4\epsilon^{\mu_n}\frac{k_n}{k_{n+1}^2}\cdot C \epsilon k_n \quad \mbox{(by inequality (\ref{integralbound1}))}\\
&\leq& 2C\epsilon^{2\mu_n}.
\end{eqnarray*}

In the previouos estimation, to go from the second to the third line, we have used the fact that 
\[
\left|\mathbf{k}-\alpha\right|\geq \frac{\left|\mathbf{k}\right|}{2}>\frac{k_{n+1}}{2}>k_n.
\]
Hence,
\[
\left|I\right|\leq \left|I_1\right|+\left|I_2\right|\leq 2C\epsilon^{2\mu_n}+2C\epsilon^{2\mu_n}=4C\epsilon^{2\mu_n}.
\]

The second term can be estimated, using inequality (\ref{integralbound1}), as follows
\begin{eqnarray*}
\left|II\right|\leq
\sum_{k_n\leq\left|\alpha\right|\leq 2\left|\mathbf{k}\right|,\left|\mathbf{k}-\alpha\right|\geq
\frac{\left|\mathbf{k}\right|}{2}}\left|\alpha 
u_{\alpha}u_{\mathbf{k}-\alpha}\right|&\leq& \frac{8\epsilon^{\mu_n}}{\left|\mathbf{k}\right|}
\sum_{k_n\leq\left|\alpha\right|\leq 2\left|\mathbf{k}\right|,\left|\mathbf{k}-\alpha\right|\geq
\frac{\left|\mathbf{k}\right|}{2}} \left|u_{\alpha}\right|\\
&\leq& 16\epsilon^{\mu_n}\cdot C\epsilon^{\mu_n}= 16C\epsilon^{2\mu_n}.
\end{eqnarray*}

To estimate the third term we proceed in a similar fashion,
\begin{eqnarray*}
\left|III\right|&\leq&
\sum_{\left|\alpha\right|\leq 2\left|\mathbf{k}\right|,k_n\leq\left|\mathbf{k}-\alpha\right|\leq
\frac{\left|\mathbf{k}\right|}{2}}
\left|\alpha u_{\alpha}u_{\mathbf{k}-\alpha}\right|\\
&\leq& 
\frac{4\epsilon^{\mu_n}}{\left|\mathbf{k}\right|}
\sum_{\left|\alpha\right|\leq 2\left|\mathbf{k}\right|,k_n
\leq\left|\mathbf{k}-\alpha\right|\leq\frac{\left|\mathbf{k}\right|}{2}}
 \left|u_{\mathbf{k}-\alpha}\right|
\leq 2C\epsilon^{2\mu_n}.
\end{eqnarray*}

To estimate the fourth term we split it as $IV=IV_1+IV_2$ and proceed,
\begin{eqnarray*}
\left|IV_1\right|&=&
\sum_{\left|\alpha\right|\leq 2\left|\mathbf{k}\right|,\left|\mathbf{k}-\alpha\right|\leq k_{-1}}
\left|\alpha u_{\alpha}u_{\mathbf{k}-\alpha}\right|\\
&\leq& \sum_{\left|\alpha\right|\leq 2\left|\mathbf{k}\right|,\left|\mathbf{k}-\alpha\right|\leq k_{-1}}
\frac{\epsilon^{\mu_n}}{\left|\alpha\right|}\left|u_{\mathbf{k}-\alpha}\right|\\
&\leq& 
\frac{2\epsilon^{\mu_n}}{k_{n+1}}\cdot D\cdot C\cdot k_{-1} \leq 
2\epsilon^{\mu_n}\frac{k_n}{k_{n+1}}\cdot D \cdot C\cdot \frac{k_{-1}}{k_n}\leq 
C\epsilon^{\mu_n}\epsilon^{2^n}\leq C\epsilon^{2\mu_n}.
\end{eqnarray*}

In the previous 
estimation, to pass 
from the first to the second line we have used the fact that $\left|\mathbf{k}-\alpha\right|\leq k_{-1}$
implies, by the triangular inequality, that 
\[
\left|\alpha\right|\geq \left|\mathbf{k}\right|-k_{-1}\geq k_{n+1}-k_{-1}\geq \frac{1}{2}k_{n+1}>k_n,
\]
and to pass from the second to the third line we made use of inequality (\ref{integralbound1}).
We estimate $IV_2$ as follows,
\begin{eqnarray*}
\left|IV_2\right|&=&
\sum_{\left|\alpha\right|\leq 2\left|\mathbf{k}\right|,k_{-1}\leq \left|\mathbf{k}-\alpha\right|\leq k_n}
\left|\alpha u_{\alpha}u_{\mathbf{k}-\alpha}\right|\\
&\leq& \sum_{\left|\alpha\right|\leq 2\left|\mathbf{k}\right|,k_{-1}\leq \left|\mathbf{k}-\alpha\right|\leq k_n}
\frac{\epsilon^{\mu_n}}{\left|\alpha\right|}\left|u_{\mathbf{k}-\alpha}\right|\\
&\leq& 
\frac{2\epsilon^{\mu_n}}{k_{n+1}}\cdot C\cdot \epsilon k_n
\leq
2\epsilon^{\mu_n}C\frac{k_n}{k_{n+1}}\epsilon \leq 2C\epsilon^{\mu_n}\epsilon^{2^n}\leq 2C\epsilon^{2\mu_n},
\end{eqnarray*}
and again we have used the fact that 
\[
\left|\mathbf{k}-\alpha\right|\leq k_n \quad \mbox{implies}\quad 
\left|\alpha\right|>k_{n+1}-k_n>\frac{1}{2}k_{n+1}, 
\]
and inequality (\ref{integralbound1}). 
Hence, we obtain the bound
\[
\left|IV\right|\leq \left|IV_1\right|+\left|IV_2\right|<3C\epsilon^{2\mu_n}.
\]

Finally, the fifth term yields
\[
\left|V\right|\leq
\sum_{\left|\alpha\right|>2\left|\mathbf{k}\right|}\left|\mathbf{k} u_{\alpha}u_{\mathbf{k}-\alpha}\right|
\leq \left|\mathbf{k}\right|\sum_{\left|\alpha\right|>2\left|\mathbf{k}\right|}
\frac{\epsilon^{\mu_n}}{\left|\alpha\right|^2}\frac{\epsilon^{\mu_n}}{\left|\mathbf{k}-\alpha\right|^2},
\]
and using that $\left|\alpha\right|>2\left|\mathbf{k}\right|$ implies that 
$\left|\mathbf{k}-\alpha\right|<\left|\mathbf{k}\right|$ and inequality (\ref{integralbound2}),
we obtain
\[
\left|V\right|\leq \epsilon^{2\mu_n}\left|\mathbf{k}\right| \sum_{\left|\alpha\right|>2\left|\mathbf{k}\right|}
\frac{1}{\left|\alpha\right|^4}\leq \epsilon^{2\mu_n}\left|\mathbf{k}\right|\cdot
 C\cdot \frac{1}{\left|\mathbf{k}\right|}\leq
 2C\epsilon^{2\mu_n}.
\]

Let $\epsilon>0$ be small enough (say $\epsilon<\frac{1}{54 C}$). Then, for $t>t_n$, we can bound
the nonlinear term as
\begin{eqnarray*}
\left|\sum_{\alpha}\alpha u_{\alpha} u_{\mathbf{k}-\alpha}\right|
&\leq& \left|I\right|+\left|II\right|+\left|III\right|+\left|IV\right|+\left|V\right|\\
&\leq& 4C\epsilon^{2\mu_n}+16C\epsilon^{2\mu_n}+2C\epsilon^{2\mu_n}+2C\epsilon^{2\mu_n}+
3C\epsilon^{2\mu_n}\\
&=& 27C \epsilon^{2\mu_n}\leq\frac{1}{2}\epsilon^{2\mu_n-1}=\frac{1}{2}\epsilon^{\mu_{n+1}} ,\quad
\mbox{if} \quad \left|\mathbf{k}\right|\geq k_{n+1},
\end{eqnarray*}
and integrating the corresponding differential 
inequality for the Fourier coefficients, it follows that
\[
\left|u_{\mathbf{k}}\left(t\right)\right|
\leq \frac{\epsilon^{\mu_n}}{\left|\mathbf{k}\right|^2}\exp\left(-\left|\mathbf{k}\right|^2\left(t-t_n\right)\right)
+\frac{1}{2}\epsilon^{\mu_{n+1}}\left(1-\exp\left(-\left|\mathbf{k}\right|^2\left(t-t_n\right)\right)
\right)\frac{1}{\left|\mathbf{k}\right|^2}
\]
\[
\mbox{as long as}\quad \left|\mathbf{k}\right|\geq k_{n+1} \quad \mbox{and} \quad t>t_{n}.
\]
This shows, using (\ref{exponentialestimate1}) and Lemma \ref{exponentialrecurrence}, that
\[
\left|u_{\mathbf{k}}\left(t\right)\right|
\leq \frac{1}{2}\epsilon^{\mu_n}\epsilon^{2^n}\cdot\frac{1}{\left|\mathbf{k}\right|^2}
+\frac{1}{2}\frac{\epsilon^{\mu_{n+1}}}{\left|\mathbf{k}\right|^2}\leq 
\frac{\epsilon^{\mu_{n+1}}}{\left|\mathbf{k}\right|^2}, 
\]
whenever 
$\left|\mathbf{k}\right|\geq k_{n+1}$ and $t>t_{n+1}$.

Assume that $t>\rho>\rho-\frac{\rho}{2^n}$ ($n\in \mathbb{N}$). Given any $\mathbf{k}$, let
$n\in \mathbb{N}$ large enough so that $k_n\leq \left|\mathbf{k}\right|<k_{n+1}$. Then,
as we just showed, the estimate
\[
\left|u_{\mathbf{k}}\left(t\right)\right|\leq \frac{\epsilon^{\mu_n}}{\left|\mathbf{k}\right|^2}
\quad \mbox{holds}.
\]
Since $\mu_{n}\geq \frac{1}{2}2^n$, it follows that 
$\epsilon^{\mu_{n}}\leq \frac{k_0^{\frac{1}{4}}}{\left|\mathbf{k}\right|^{\frac{1}{4}}}$
whenever $\left|\mathbf{k}\right|\leq k_{n+1}$,
and hence that at time $t>\rho$, for wave numbers $\mathbf{k}\in \mathbb{Z}^3$ large enough,
\[
\left|u_{\mathbf{k}}\left(t\right)\right|\leq \frac{C}{\left|\mathbf{k}\right|^{2.25}} 
\quad\mbox{holds}.
\]

This shows that $u\in L^{\infty}\left(\rho,T;H^{\frac{1}{2}+\frac{1}{8}}\left(\mathbb{T}^3\right)\right)$,
and from the work of Leray (see \cite{Leray}, and \cite{Doering} for a proof of this regularity
result in the periodic case) the Theorem follows.

\hfill $\Box$

\bigskip
{\it Remark. } One can end the proof of the previous Theorem without recurring to Leray's regularity result. Indeed,
it can be shown that if there is a $C$ so that 
\[
\left|u_{\mathbf{k}}\left(t\right)\right|\leq \frac{C}{\left|\mathbf{k}\right|^{2+\rho}}\quad \mbox{for all}
\quad t\in \left(t_{*},T\right)
\]
then for every $\delta>0$ there exists a $D$ such that
\[
\left|u_{\mathbf{k}}\left(t\right)\right|\leq 
\frac{D}{\left|\mathbf{k}\right|^{2+2\rho}} \quad \mbox{for} \quad t\in\left(t_{*}+\delta, T\right).
\] 

Hence, by a finite iterarion we arrive to the fact that the enstrophy remains uniformly bounded on any
interval of the form $\left(t_{*}+\delta, T\right)$ with $\delta>0$. From
this, one can use the method of Mattingly-Sinai (\cite{Sinai}) to conclude that $u$ is analytic in space.

\section{Proof of Theorem \ref{Escauriaza}}
\label{Seregin}

We proceed now with the proof of Theorem \ref{Escauriaza}.
Our first important observation is the following simple,
\begin{lemma}
\label{basicestimate}
Assume $\int \left|\left|\nabla\right|^{\frac{1}{2}} u\right|^2<M$. Then there is a constant 
$c\geq 1$, 
\[
\sum_{0<\left|\alpha\right|\leq r}\left|u_{\alpha}\right|\leq cM^{\frac{1}{2}}r.
\]
\end{lemma}
\begin{proof}
The result follows from a judicious application of the Cauchy-Schwartz inequality.
Indeed,
\[
\sum_{0<\left|\alpha\right|\leq r}\left|u_{\alpha}\right|\leq 
\left(\sum_{0<\left|\alpha\right|\leq r}\left|\alpha\right|u_{\alpha}^2\right)^{\frac{1}{2}}
\left(\sum_{0<\left|\alpha\right|\leq r}\frac{1}{\left|\alpha\right|}\right)^{\frac{1}{2}}\leq cM^{\frac{1}{2}}r,
\]
where $c$ is a constant such that 
\[
\sum_{0<\left|\alpha\right|\leq r} \frac{1}{\left|\mathbf{k}\right|}\leq 
c\int_{1\leq\left|\mathbf{x}\right|\leq r, \,\mathbf{x}\in \mathbf{R}^3}\frac{1}{\left|\mathbf{x}\right|}\,d\mathbf{x}.
\]

Notice that $c$ does not depend on $r$.

\hfill
\end{proof}

\bigskip

\begin{lemma}
\label{fundamental}
Let $u\left(x,t\right)\in L^{\infty}\left(0,T;H^{\frac{1}{2}}\left(\mathbb{T}^3\right)\right)$ be
a solution of (\ref{Navierstokes}). There exists a $\delta>0$ such that for all $\rho>0$, if
$\left\|u\left(t\right)\right\|_{H^{\frac{1}{2}}\left(\mathbb{T}^3\right)}< \delta$ for $t\in\left(0,T\right)$, then
there is a $K_0=K_0\left(\left\|u\right\|_{L^{\infty}\left(0,T;H^{\frac{1}{2}}\left(\mathbb{T}^3\right)\right)},\rho\right)$ 
and there is an $N=N\left(\left\|u\right\|_{L^{\infty}\left(0,T;H^{\frac{1}{2}}\left(\mathbb{T}^3\right)\right)}\right)$
such that if $\left|\mathbf{k}\right|\geq K_0 2^n$,
$n\geq N$, then
\begin{equation}
\label{fundamentalnorm}
\left|u_{\mathbf{k}}\left(t\right)\right|\leq \frac{\sqrt{2}\left(4+\sqrt{2}\right)L_0/\left(1-2c\sqrt{L_0}\right)}
{\left|\mathbf{k}\right|^{2-\gamma_n}},
\quad \gamma_n=\frac{1}{2^n}\quad
if \quad t\geq t_{n},
\end{equation}
where 
$t_n=\rho-\frac{\rho}{2^n}$ and
$\sqrt{L_0}=\sup_{t\in\left(0,T\right)} \left\|u\left(t\right)\right\|_{H^{\frac{1}{2}}\left(\mathbb{T}^3\right)}$
and $c$ as in Lemma \ref{basicestimate}.
\end{lemma}
\begin{proof}
First, given $\rho>0$ choose $K_0$ large enough so that for all $n$, and $l>n$,
\begin{equation}
\label{exponentialsmallness2}
\exp\left(-\left(K_0\cdot 2^l\right)^2\cdot \frac{\rho}{2^n}\right)< \frac{\sqrt{L_0}}{\left(K_0 2^l\right)^2}.
\end{equation}

Now we proceed to show estimate (\ref{fundamentalnorm}) by induction. We begin by showing the result for $n=1$.
Write,
\[
\sum_{\alpha} \alpha u_{\alpha}u_{\mathbf{k}-\alpha}=I+II
\]
where the meaning of the terms on the right hand side will become clear in what follows. We
begin by estimating $I$ as
\begin{eqnarray*}
\left|I\right|&=&
\left|\sum_{\left|\alpha\right|<2\left|\mathbf{k}\right|} \alpha u_{\alpha} u_{\mathbf{k}-\alpha}\right|
\leq
\sum_{\left|\alpha\right|<2\left|\mathbf{k}\right|} 
\left|\alpha u_{\alpha} u_{\mathbf{k}-\alpha}\right|\\
&\leq& \sqrt{2}\sqrt{\left|\mathbf{k}\right|}
\sum_{\left|\alpha\right|<2\left|\mathbf{k}\right|}
\left|\sqrt{\left|\alpha\right|}u_{\alpha}u_{\mathbf{k}-\alpha}\right|,
\end{eqnarray*}
and then the Cauchy-Schwarz inequality yields,
\[
\left|I\right|\leq \sqrt{2}\sqrt{\left|\mathbf{k}\right|}
\left(\sum_{\left|\alpha\right|<2\left|\mathbf{k}\right|}\left|\alpha\right|u_{\alpha}^2\right)^{\frac{1}{2}}
\left(\sum_{\left|\alpha\right|<2\left|\mathbf{k}\right|}u_{\mathbf{k}\alpha}^2\right)^{\frac{1}{2}}
\leq 
\sqrt{2}\sqrt{\left|\mathbf{k}\right|}\left\|u\right\|_{H^{\frac{1}{2}}\left(\mathbb{T}^3\right)}
\left\|u\right\|_{L^2\left(\mathbb{T}^3\right)}.
\]

We estimate $II$ as follows,
\begin{eqnarray*}
\left|II\right|&=&
\left|\sum_{\left|\alpha\right|\geq 2\left|\mathbf{k}\right|}\alpha u_{\alpha} u_{\mathbf{k}-\alpha}\right|\\
&\leq&
\sqrt{2}\sum_{\left|\alpha\right|\geq 2\left|\mathbf{k}\right|}\sqrt{\left|\alpha\right|} \left|u_{\alpha}\right|
 \sqrt{\left|\mathbf{k}-\alpha\right|}\left|u_{\mathbf{k}-\alpha}\right|
\leq  \sqrt{2}\left\|u\right\|^2_{H^{\frac{1}{2}}\left(\mathbb{T}^3\right)}.
\end{eqnarray*}
The last inequality in the previous estimation follows, once again, from Cauchy-Schwarz.

Recall that
\[
\sqrt{L_0}=\sup_{t\in\left(0,T\right)}\left\|u\left(t\right)\right\|_{H^{\frac{1}{2}}\left(\mathbb{T}^3\right)}.
\]

Then, the previous estimates, by integrating the differential inequality obtained for the 
Fourier coefficients, yield the following bound
\[
\left|u_{\mathbf{k}}\left(t\right)\right|\leq \sqrt{L_0}\exp\left({-\frac{\rho\left|\mathbf{k}\right|^2}{2}}\right)+
\frac{\sqrt{2}L_0}{\left|\mathbf{k}\right|^{\frac{3}{2}}}
+\frac{\sqrt{2}L_0}{\left|\mathbf{k}\right|^2}
\leq \frac{2 L_0}{\left|\mathbf{k}\right|^{\frac{3}{2}}}
\]
if $\left|\mathbf{k}\right|\geq K_0$, for $K_0$ large enough, as long as $t>\rho-\frac{\rho}{2}=\frac{\rho}{2}$.
Here we have used the fact that for a given $w\in H^{\frac{1}{2}}\left(\mathbb{T}^3\right)$, if $w_0=0$ then 
$\left\|w\right\|_{L^2\left(\mathbb{T}^3\right)}\leq \left\|w\right\|_{H^{\frac{1}{2}}\left(\mathbb{T}^3\right)}$.

Now assume that for $t>t_n$ the estimate holds for
\[
\left|u_{\mathbf{k}}\left(t\right)\right|\leq \frac{D_n}{\left|\mathbf{k}\right|^{2-\gamma_n}}
\quad \mbox{holds for} \quad \left|k\right|\geq K_0 2^n.
\]
We will show that if $t>t_{n+1}$, $n\geq 1$, then the following estimate holds
\[
\left|u_{\mathbf{k}}\left(t\right)\right|\leq \frac{D_{n+1}}{\left|k\right|^{2-\gamma_{n+1}}}
\quad\mbox{if} \quad
\left|\mathbf{k}\right|\geq K_0 2^{n+1},
\]
where (with $c$ as in Lemma \ref{basicestimate})
\begin{equation}
\label{basicrecurrence}
D_{n+1}=2c\sqrt{L_0}D_n + \left(4+\sqrt{2}\right)L_0,
\quad D_0=2L_0.
\end{equation}
In order to do this, again we must estimate
\[
\sum_{\alpha}\alpha u_{\alpha}u_{\mathbf{k}-\alpha}= I+II+III
\]
where
\[
I=\sum_{\left|\alpha\right|\leq 2\left|\mathbf{k}\right|, \left|\mathbf{k}-\alpha\right|\geq 
\frac{1}{2}\left|\mathbf{k}\right|^{1-\gamma_n}}\alpha u_{\alpha}u_{\mathbf{k}-\alpha},\quad
II=\sum_{\left|\mathbf{k}-\alpha\right|\leq \frac{1}{2}\left|\mathbf{k}\right|^{1-\gamma_n}}\alpha
 u_{\alpha}u_{\mathbf{k}-\alpha}
\]
and
\[
III=\sum_{\left|\alpha\right|\geq 2\left|\mathbf{k}\right|}\alpha u_{\alpha}u_{\mathbf{k}-\alpha}.
\]

From now on, the assumption $\left|\mathbf{k}\right|>K_02^{n+1}$ is in place.
Let us estimate $I$. First notice that if $\left|\alpha\right|\leq 2\left|\mathbf{k}\right|$
and $\left|\mathbf{k}-\alpha\right|\geq \frac{1}{2} \left|\mathbf{k}\right|^{1-\gamma_n}$ then 
the following inequalities hold
\[
\left|\mathbf{k}-\alpha\right|\geq \frac{1}{2}\left|\mathbf{k}\right|^{1-\gamma_n}\geq 
\frac{1}{2}\left(\frac{1}{2}\right)^{1-\gamma_n}\left|\alpha\right|^{1-\gamma_n}
\geq \frac{1}{4}\left|\alpha\right|^{1-\gamma_n}.
\]

By taking square roots, from the previous inequalities we deduce that
\[
2\sqrt{\left|\mathbf{k}-\alpha\right|}
\geq \left|\alpha\right|^{\frac{1}{2}-\gamma_{n+1}},
\]
and hence if $\left|\alpha\right|\leq 2\left|\mathbf{k}\right|$ we get
\[
\sqrt{\left|\alpha\right|}\leq 2\left|\alpha\right|^{\gamma_{n+1}}\sqrt{\left|\mathbf{k}-\alpha\right|}
\leq 2 \cdot 2^{\gamma_{n+1}}\left|\mathbf{k}\right|^{\gamma_{n+1}}\sqrt{\left|\mathbf{k}-\alpha\right|}.
\]
Therefore, $I$ can be estimated as
\[
\left|I\right|\leq 2\cdot 2^{\gamma_{n+1}}\left|\mathbf{k}\right|^{\gamma_{n+1}}
\sum \sqrt{\left|\alpha\right|}\left|u_{\alpha}\right|\sqrt{\left|\mathbf{k}-\alpha\right|}
\left|u_{\mathbf{k}-\alpha}\right| \leq
4L_0 \left|\mathbf{k}\right|^{\gamma_{n+1}},
\]
where the last inequality follows from the Cauchy-Schwarz inequality.

Let us estimate $II$.
From the induction hypothesis, we find that
\[
\left|II\right|\leq \frac{D_n}{\left(1-\frac{1}{2}\right)^{1-\gamma_n}\left|\mathbf{k}\right|^{1-\gamma_n}}
\sum_{\left|\mathbf{k}-\alpha\right|\leq 
\frac{1}{2}\left|\mathbf{k}\right|^{1-\gamma_n}} \left|u_{\mathbf{k}-\alpha}\right|
\leq \frac{c\frac{1}{2}}{1-\frac{1}{2}}\sqrt{L_0}D_n=c\sqrt{L_0}D_n,
\]
since 
$\left|\mathbf{k}-\alpha\right|\leq \frac{1}{2}\left|\mathbf{k}\right|^{1-\gamma_n}$
implies that 
\[
\left|\alpha\right|\geq \left(1-\frac{1}{2}\right)\left|\mathbf{k}\right|>K_0 2^n
\]
(i.e., the first inequality follows from the induction hypothesis and the assumption 
$\left|\mathbf{k}\right|>K_0 2^{n+1}$).
The last inequality in the estimation of $II$ follows from Lemma \ref{basicestimate}.

Finally, again by the Cauchy-Schwarz inequality, we have that
\[
\left|III\right| \leq\sum_{\left|\alpha\right|>2\left|\mathbf{k}\right|}
\sqrt{\left|\alpha\right|}\sqrt{2\left|\mathbf{k}-\alpha\right|}\left|u_{\alpha}\right|
\left|u_{\mathbf{k}-\alpha}\right| \\
\leq \sqrt{2}L_0.
\]
Hence, integrating the ODE system for the Fourier coefficients yields,
\begin{eqnarray*}
\left|u_{\mathbf{k}}\left(t\right)\right|
&\leq&
\sqrt{L_0}\exp\left(-\left|\mathbf{k}\right|^2\left(t-t_n\right)\right)\\
&&+\left(\frac{c\sqrt{L_0}D_n}{\left|\mathbf{k}\right|^2}+\frac{\sqrt{2}L_0}{\left|\mathbf{k}\right|^2}
+\frac{4L_0}{\left|\mathbf{k}\right|^{2-\gamma_{n+1}}}\right)
\left(1-\exp\left(-\left|\mathbf{k}\right|^2\left(t-t_n\right)\right)
\right),
\end{eqnarray*}
for
\[
\left|\mathbf{k}\right|\geq K_0 2^{n+1} \quad\mbox{and}\quad t>t_n.
\]

Therefore,
by our choice of $K_0$ (given by (\ref{exponentialsmallness2})), we obtain the following estimate, valid for all
frequencies $\mathbf{k}$ such that $\left|\mathbf{k}\right|\geq K_0 2^{n+1}$ and $t>t_{n+1}$,
\[
\left|u_{\mathbf{k}}\left(t\right)\right|< \frac{2c\sqrt{L_0} D_n+\left(4+\sqrt{2}\right)L_0}
{\left|\mathbf{k}\right|^{2-\gamma_{n+1}}}.
\]

Finally, to prove estimate (\ref{basicrecurrence}), 
all that is left to show is how to bound the sequence $\left(D_n\right)_{n=0,1,2,\dots}$ effectively
for large $n$.
This can be done as long as $2c\sqrt{L_0}<1$, and it can be seen in this case that
\[
\lim_{n\rightarrow \infty} D_n= \frac{\left(4+\sqrt{2}\right)L_0}{1-2c\sqrt{L_0}}.
\]
This finishes the proof of the Lemma (take $\delta\leq\frac{1}{2c}$).

\hfill
\end{proof}

Theorem \ref{Escauriaza} follows immediatly from the previous Lemma. Indeed,

\bigskip

{\bf{\it Proof of Theorem \ref{Escauriaza}. }}\normalfont Since for $n$ large enough 
$K_0^{\gamma_n}2^{\frac{n}{2^n}}\leq \sqrt{2}$, if $\mathbf{k}$ satisfies
that $K_02^n\leq\left|\mathbf{k}\right|<K_0 2^{n+1}$ and $t>t_{n+1}$, Lemma \ref{fundamental} yields
\[
\left|u_{\mathbf{k}}\left(t\right)\right|\leq 
\frac{\left|\mathbf{k}\right|^{\gamma_n}\cdot\sqrt{2}\left(4+\sqrt{2}\right)L_0/\left(1-2c\sqrt{L_0}\right)}
{\left|\mathbf{k}\right|^2}
\leq
\frac{2\left(4+\sqrt{2}\right)L_0/\left(1-2c\sqrt{L_0}\right)}{\left|\mathbf{k}\right|^2}.
\]
Therefore if $t>\rho$, for wave numbers $\mathbf{k}$ large enough the required estimate holds, and Theorem \ref{Escauriaza}
 is proved. 

\hfill $\Box$

\section{Last Remarks}
\label{lastremarks}

\subsection{Corollary from the Proof of Theorem \ref{Escauriaza}}\label{Seregin2}
Notice that if there is a $k_0$ so that for all $t\in\left(t_0,t_1\right)$ the expression
\begin{equation}
\label{smallgap}
\sum_{\left|\alpha\right|\geq k_0} \left|\alpha\right|\left|u_{\alpha}\right|^2<\delta 
\quad\mbox{for} \quad \delta \quad\mbox{small enough}
\end{equation}
then it follows from the proof of Theorem \ref{Escauriaza} that for any $\rho>0$ there is a $K_1$ 
which depends on $\rho$ (but not on $\delta>0$ for small $\delta$), such that
\[
\sup_{\left|\mathbf{k}\right|\geq K_1}\left|\mathbf{k}\right|^2\left|u_{\mathbf{k}}\left(t\right)\right|\leq C\delta
\quad \mbox{for} \quad t\in\left(t_0+\rho,t_1\right),
\]
and the constant $C$ is also independent of $\delta>0$ for $\delta$ small.
From Theorem \ref{uniqueness} it 
then follows that in the case that (\ref{smallgap}) holds on $\left(t_0,t_1\right)$,
for $\delta>0$ small enough,
then $u$ is a smooth solution of the Navier-Stokes equation on $\left(t_0,t_1\right)$.
On the other hand (\ref{smallgap}) holds
whenever $u\in C\left(\left(t_0,t_1\right),H^{\frac{1}{2}}\left(\mathbb{T}^3\right)\right)$,
and hence $C\left(\left(t_0,t_1\right),H^{\frac{1}{2}}\left(\mathbb{T}^3\right)\right)$ is
a regularity class for the Navier-Stokes equation
(this is a classical result due to Giga in \cite{Y. Giga} and v. Wahl in \cite{Wahl}).
 Also, there is an $\eta>0$ such that (\ref{smallgap}) holds whenever $u$ satisfies de following
property
\[
\limsup_{t\rightarrow t_2^-}
\left\|u\left(t\right)-u\left(t_2\right)\right\|_{H^{\frac{1}{2}}\left(\mathbb{T}^3\right)}<\eta,
\quad \mbox{for all} \quad t_2\in \left(t_0,t_1\right),
\]
and hence, in such a case ("small discontinuities in $H^{\frac{1}{2}}\left(\mathbb{T}^3\right)$ are
allowed"), $u$ is smooth.

This last assertion should be compared with the recent results in \cite{Cheskidov},
and the main result in \cite{Sohr}.

\subsection{On Corollary \ref{globalsolution}}\label{corolario1}
As we said in the introduction, among the hypothesis of Corollary \ref{globalsolution}, 
it is not needed to have $\psi$, the initial condition, to be real-valued nor we need
to use Leray-Hopf's Existence Theorem (as in the proof we have suggested in the introduction).
Indeed, from the proof of Theorem \ref{smallnorm},
it is easy to see that there is an $\epsilon>0$ such that if the initial condition $\psi$
(real-valued or not)
satisfies 
\[
\left\|\psi\right\|_2< \epsilon,
\]
then the solution to any finite dimensional Galerkin approximation
to (\ref{Navierstokes}) has $\Phi\left(2\right)$ norm smaller than $\epsilon$. Therefore, 
the $L^2\left(\mathbb{T}^3\right)$-norms of the Galerkin approximations to 
(\ref{Navierstokes}) remain uniformly bounded, and hence, using the 
weak compactness of $L^2\left(\mathbb{T}^3\right)$, 
we obtain
a weak solution to the Navier-Stokes 
system (in the sense of Leray-Hopf, except that the energy inequality may not hold), 
and by the proof of Theorem \ref{uniqueness} (see Remark after the end of the proof of Theorem
\ref{uniqueness}, and also take into account that
if $u\left(x,t\right)$ condition has $\Phi\left(2\right)$-norm uniformly small in time, then
the energy inequality, which prevents backscattering, is not needed in the proof of Theorem \ref{uniqueness}),
 this will be a strong solution of (\ref{Navierstokes}).


\begin{thebibliography}{}


%
 
 
%

\bibitem[1]{Arnold} M.D. Arnold, Ya. G. Sinai, Global Existence and Uniqueness Theorem
for 3D-Navier Stokes System on $\mathbb{T}^3$ for small initial conditions in the spaces
$\Phi\left(\alpha\right)$, arXiv:0710.3842v1, to appear in Pure and Applied Mathematics 
Quarterly {\bf 4}, No 1, 1--9.

\bibitem[2]{Escauriaza} L. Escauriaza, G. Seregin, V. Sverak, $L_{3,\infty}$-solutions of
Navier-Stokes equations and backward uniqueness, Uspekhi Mat Nauk. {\bf 58} (2003), No 2 (350),3-44,
Translation in Russian Math. Surveys {\bf 58} (2003), No. 2, 211-250.

\bibitem[3]{Cheskidov} A. Cheskidov, R. Shvydkoy, On the regularity of weak solutions of the 3D Navier-Stokes 
equations in $B_{\infty,\infty}^{-1}$, arXiv:0708.3067v2.

\bibitem[4]{Doering} C. R. Doering, J. D. Gibbon, Applied Analysis of the Navier-Stokes equations,
Cambridge Texts in Applied Mathematics, Cambridge University Press, 1995. 

\bibitem[5]{Y. Giga} Y. Giga, Solutions for semilinear parabolic equations in $L^p$ and regularity 
of weak solutions of the Navier-Stokes system, J. Differential Equations {\bf 62} (1986), No 2,
186--212.

\bibitem[6]{Kaloshin} The result by V. Kaloshin and Yu. Sannikov is mentioned in Weinan E and Ya.G. Sinai
Recent results in mathematical and statistical hydrodynamics, Russ. Math. Surveys, 55:4,
635-666, 2000

\bibitem[7]{Sohr} H. Kozono and H. Sohr, Regularity criterion of weak solutions to the Navier-Stokes equations,
Adv. Differential Equations {\bf 2} (1997), No 4, 535--554.

\bibitem[8]{Leray} J. Leray, Sur le mouvement d'un liquide visqueux emplissant l'espace, Acta Math. {\bf 63} 
(1934), 193-248 

\bibitem[9]{Sinai} J. Mattingly, Ya. G. Sinai, An elementary proof of the existence and uniqueness theorem
for the Navier Stokes equation, Commun. Contemp. Math. {\bf 1} (1999), No 4, 497-516.

\bibitem[10]{Temam} R. Temam, Navier-Stokes Equations and Nonlinear Functional Analysis,
CBMS-NSF Regional Conference Series in Applied Mathematics, Society for Industrial and Applied
Mathematics, 1983.

\bibitem[10]{Wahl}  W. von Wahl, Regularity of weak solutions of the Navier-Stokes equations, Proc. Symp. Pure
Appl. Math., {\bf 45} (1986), 497-503.
\end{thebibliography}
\end{document}